\documentclass[a4paper,12pt]{amsart}
\usepackage{amssymb}
\usepackage{amsmath}
\usepackage[all]{xy}

\usepackage{hyperref}

\theoremstyle{plain}
{
  \newtheorem{thm}{Theorem}[section]
  
  \newtheorem{cor}[thm]{Corollary}
  \newtheorem{lem}[thm]{Lemma}

}

\DeclareMathOperator{\et}{\text{\it \'et}}

\DeclareMathOperator{\GL}{GL}

\newcommand{\Spec}{\text{\rm Spec}\,}

\DeclareMathOperator{\Frac}{Frac}
\DeclareMathOperator{\corad}{corad}

\usepackage{hyperref}

\begin{document}

\title[Non-stable $K_1$-functors of discrete valuation rings]
{Non-stable $K_1$-functors of discrete valuation rings containing a field}

\author{Philippe Gille}\thanks{P. Gille was supported by the project "Group schemes, root systems, and
related representations" founded by the European Union - NextGenerationEU through
Romania's National Recovery and Resilience Plan (PNRR) call no. PNRR-III-C9-2023-
I8, Project CF159/31.07.2023, and coordinated by the Ministry of Research, Innovation and Digitalization (MCID)
of Romania. }
\address{UMR 5208 Institut Camille Jordan - Universit\'e Claude Bernard Lyon 1
43 boulevard du 11 novembre 1918
69622 Villeurbanne cedex - France}  
\email{gille@math.univ-lyon1.fr}
\address{and Institute of Mathematics "Simion Stoilow" of the Romanian Academy,
21 Calea Grivitei Street, 010702 Bucharest, Romania}

\author{Anastasia Stavrova}
\thanks{}
\address{St. Petersburg Department of Steklov Mathematical Institute, nab. r. Fontanki 27, 191023 St. Petersburg, Russia}
\email{anastasia.stavrova@gmail.com}

\maketitle

\begin{abstract}
Let $k$ be a field, and let $G$ be a simply connected semisimple $k$-group which is isotropic and contains a
stricty proper parabolic $k$-subgroup $P$.
Let $D$ be a discrete valuation ring which is a local ring of a smooth algebraic curve over $k$.
We show that $K_1^G(D)=K_1^G(K)$,
where $K$ is the fraction field of $D$ and
$K_1^G(-)=G(-)/E_P(-)$ is the corresponding
non-stable $K_1$-functor, also called the Whitehead group of $G$.
As a consequence, $K_1^G(D)$ coincides with the (generalized) Manin's $R$-equivalence class group of $G(D)$.
\end{abstract}

\section{Introduction}
Let $R$ be a commutative ring with 1. Let $G$ be a reductive group scheme over $R$ in the sense of~\cite{SGA3}.
For any reductive group $G$ over $R$ and a parabolic subgroup $P$ of $G$,
one defines the elementary subgroup $E_P(R)$ of $G(R)$ as the subgroup
generated by the $R$-points of the unipotent radicals of $P$ and of any opposite parabolic $R$-subgroup
$P^-$, and considers the
corresponding non-stable $K_1$-functor $K_1^{G,P}(R)=G(R)/E_P(R)$~\cite{PS,St-poly,St-k1}. It does not depend on the
choice of $P^-$ by~\cite[Exp. XXVI Cor. 1.8]{SGA3}.
In particular, if $A=k$ is a field and $P$ is minimal, $E(k)$ is nothing but the group $G(k)^+$ introduced
by J. Tits~\cite{Tits64}, and $K_1^G(k)$ is the subject of the Kneser--Tits problem~\cite{Gi}. If $G=\GL_n$
and $P$ is a Borel subgroup,
then $K_1^G(R)=\GL_n(R)/E_n(R)$, $n\ge 1$, are the usual non-stable $K_1$-functors of algebraic $K$-theory.

A parabolic $R$-subgroup $P$ in $G$ is called
strictly proper, if it intersects properly every non-trivial semisimple normal $R$-subgroup of $G$.
If $R$ is semilocal (or, more generally, a local-global ring, see~\cite{GiNe}), then $E_P(R)$ is the same
for all strictly proper parabolic $R$-subgroups $P$ of $G$~\cite[Th. 2.1]{St-poly}, and, in particular, is normal in $G(R)$.

Let $R$ be a regular local ring and $K$ be its field of fractions. The Serre--Grothendieck
conjecture (\cite[Remarque, p.31]{Se},~\cite[Remarque 3, p.26-27]{Gr1}, and~\cite[Remarque~1.11.a]{Gr2}) predicts that
for any reductive group $G$ over $R$,
the natural map between the first non-abelian \'etale cohomology sets
\begin{equation}\label{eq:H1}
H^1_{\et}(R,G)\to H^1_{\et}(K,G)
\end{equation}
has trivial kernel. It is known to hold in many cases, in particular, for all regular local rings $R$ which are
equicharacteristic, i.e. contain a field~\cite{Pan3}, and for all isotropic reductive groups over rings which are geometrically regular over
a Dedekind ring~\cite{CF}. The \'etale cohomology functor $H^1_{\et}(-,G)$ is often
viewed as a non-stable, non-abelian version of the $K_0$-functor of algebraic $K$-theory, and
the map~\eqref{eq:H1} itself is reminiscent of the first term in the Gersten conjecture.
Building on this analogy, it is natural to ask if the map
$$
K_1^{G,P}(R)\to K_1^{G,P}(K)
$$
is injective under the same assumptions on $R$.

We say that $G$ has isotropic rank $\ge n$ over $R$, if every non-trivial semisimple normal subgroup of $G$
contains an $R$-subgroup of the form $\mathbb{G}^n_{m,R}$. It was proved in~\cite{St-k1} that
if $G$ has isotropic rank $\ge 2$ over $R$,
and $R$ is any regular local ring containing a field, then
\begin{equation}\label{eq:K1}
K_1^{G,P}(R)\to K_1^{G,P}(K)
\end{equation}
is injective, where $K$ is the fraction field of $R$ and $P$ is a strictly proper parabolic $R$-subgroup.
However, the method of proof was definitely inapplicable to groups of isotropic rank 1. Namely, it relied on the injectivity
of the map
$$
K_1^{G,P}(R[x])\to K_1^{G,P}(R[x]_f),
$$
where $f\in R[x]$ is a monic polynomial; and this injectivity
is false for $G=SL_2$, $f=x$ and $R$ any discrete valuation ring~\cite{Chu}\footnote{In fact,
if $B$ is a standard Borel subgroup of $SL_2$ and $R$ is a discrete valuation ring, then
$K_1^{SL_2,B}(R[x])$ is not a group and, in particular, is not trivial~\cite[Proposition 1.8]{Chu},
while $K_1^{SL_2,B}(R[x]_x)=1$~\cite[Theorem 3.1]{Chu}.} Thus, the injectivity of~\eqref{eq:K1}
is not known in general even for equicharacteristic discrete valuation rings (although for $SL_2$, as well
as for any other simply connected split group, it is trivially true).

In~\cite{GSt} we extended Manin's notion of $R$-equivalence of points from algebraic varieties over fields to schemes
over commutative rings. The very definition of the $R$-equivalence implies
that $K_1^{G,P}(R)=G(R)/E_P(R)$ surjects onto the $R$-equivalence class group $G(R)/\mathcal{R}$
for any $R$~\cite[2.1, 4.3]{GSt}. It has been previously known that for any field $K$ and any simply connected semisimple group $G$ having a strictly parabolic
$K$-subgroup one has $K_1^{G,P}(K)=G(K)/\mathcal{R}$~\cite{Gi}. In~\cite[Prop. 8.10]{GSt} we have extended this
equality to henselian discrete valuation rings and concluded that
the map~\eqref{eq:K1} is an isomorphism for every simply connected semisimple group $G$ and every henselian discrete
valuation ring $R$. In the present paper we obtain the following results in the non-henselian case.

\begin{thm}\label{thm:main}
Let $k$ be a field, let $G$ be a reductive algebraic group over $k$ having a strictly proper parabolic
$k$-subgroup $P$. Let $D$ be a discrete valuation ring which is a local ring at a closed point
of a smooth algebraic curve over $k$, and let
$K$ be the fraction field of $D$. Then $K_1^{G,P}(D)\to K_1^{G,P}(K)$ is injective.
\end{thm}

\begin{cor}\label{cor:sep}
Let $k$ be a field, let $G$ be a reductive algebraic group over $k$ having a strictly proper parabolic
$k$-subgroup $P$. Let $D$ be a discrete valuation ring which is a local ring of a finitely generated $k$-algebra
and assume that the residue field of $D$ is a separable extension of $k$. Let
$K$ be the fraction field of $D$. Then $K_1^{G,P}(D)\to K_1^{G,P}(K)$ is injective.
\end{cor}

Corollary~\ref{cor:sep} is an easy consequence of Theorem~\ref{thm:main}, since one may show that any $D$ as in this
Corollary is in fact a local ring of a smooth curve over a suitable transcendental extension of $k$. This is well-known to
specialists, however, we were unable to find an exact reference, so we provide a proof in the end of the paper.

\begin{cor}\label{cor:sc-R}
Let $k$, $G$, $P$ and $D$ be as in Theorem~\ref{thm:main}, or as in Corollary~\ref{cor:sep}. Assume, moreover, that
$G$ is a simply connected semisimple group. Then $K_1^{G,P}(D)\cong G(D)/\mathcal{R}$.
\end{cor}
\begin{proof}
As mentioned above, there is a surjective homomorphism
$K_1^{G,P}(D)\to G(D)/\mathcal{R}$. Since $K_1^{G,P}(D)\to K_1^{G,P}(K)$ is injective and $K_1^{G,P}(K)\to G(K)/\mathcal{R}$
is an isomorphism~\cite{Gi}, we conclude that $K_1^{G,P}(D)\to G(D)/\mathcal{R}$ is injective. Hence the claim.
\end{proof}

The surjectivity of the map $K_1^{G,P}(D)\to K_1^{G,P}(K)$ may hold only for simply connected semisimple groups,
since it fails, for example, for $D=k[[t]]$ and $G=GL_n$ or $PGL_n$. We settle it in this generality.

\begin{thm}\label{thm:surj}
Let $D$ be any discrete valuation ring, let $K$ be the fraction field of $D$, and
let $G$ be a simply connected semisimple group scheme over $D$ having a strictly proper parabolic
$D$-subgroup scheme $P$. Then the natural map $K_1^{G,P}(D)\to K_1^{G,P}(K)$ is surjective.
\end{thm}

\section{Preliminary lemmas}

     \begin{lem}\label{lem:Ui}
Let $B$ be a Noetherian commutative ring, let $G$ be a reductive group over $B$, let $P$ be a parabolic $B$-subgroup of $G$ and
let $U_P$ be the unipotent radical of $P$. Let $A$ be a commutative $B$-algebra and let $C_1,C_2\subseteq A$ be two
$B$-subalgebras of $A$ such that $A=C_1+C_2$. Then $U_P(A)=U_P(C_1)\cdot U_P(C_2)$.
\end{lem}
\begin{proof}
By~\cite[Exp. XXVI, Prop. 2.1]{SGA3} there is a sequence
$$
U_0=U_P\supset U_1\supset U_2\supset\ldots U_n\supset \ldots
$$
of $B$-subgroup schemes of $U_P$ with the following properties.
\begin{enumerate}
\item Each $U_i$ is $B$-smooth, connected and a closed characteristic
subgroup of $P$.
\item For every $B$-algebra $R$ one has $[U_i(R),U_j(R)]\subseteq U_{i+j+1}(R)$ for all $i,j\ge 0$.
\item For all $i\ge 0$ there is a finitely generated projective $B$-module $V_i$ and an isomorphism of $B$-group schemes

$U_i/U_{i+1}\cong W(V_i)$, where $W(V_i)$ is the canonical affine $B$-scheme corresponding to $V_i$ in the sense of~\cite{SGA3}.
\item One has $U_i=1$, as soon as $i>\dim_s((U_P)_s)$ for all $s\in\Spec(B)$.
\end{enumerate}
Note that since $B$ is Noetherian, the last property implies that there is a finite integer $N\ge 0$ such that
$U_{N+1}=1$.
Taken together, these properties also imply that $H^1_{\et}(R,U_i)=0$ for any commutative $B$-algebra $R$ and $i\ge 0$,
see the proof of~\cite[Exp. XXVI, Cor. 2.2]{SGA3}. As a consequence, there are short exact sequences of groups
$$
1\to U_{i+1}(R)\to U_i(R)\to (U_i/U_{i+1})(R)\to 1.
$$

We prove that $U_i(A)=U_i(C_1)\cdot U_i(C_2)$ by descending induction on $i$. If $i=N+1$, this equality is clear. Assume
it holds for $U_{i+1}$ and prove it for $U_i$.
We have
$$
(U_i/U_{i+1})(A)\cong V_i\otimes_B A=V_i\otimes_B C_1+V_i\otimes_B C_2=(U_i/U_{i+1})(C_1)\cdot(U_i/U_{i+1})(C_2).
$$
Then
$$
U_i(A)=U_{i}(C_1)\cdot U_{i}(C_2)\cdot U_{i+1}(A).
$$
By the inductive assumption, we have $U_{i+1}(A)=U_{i+1}(C_2)U_{i+1}(C_1)$ (note that $C_1,C_2$ are interchangeable
in all statements). Then
$$
U_i(A)=U_{i}(C_1)\cdot U_{i}(C_2)\cdot U_{i+1}(A)=U_i(C_1)U_i(C_2)U_{i+1}(C_1)
$$
$$\subseteq
U_i(C_1)U_{i+1}(C_1)[U_{i+1}(C_1),U_i(C_2)]U_i(C_2)
\subseteq U_i(C_1)U_{2i+2}(A)U_i(C_2)
$$
$$
=U_i(C_1)U_{2i+2}(C_1)U_{2i+2}(C_2)U_{i}(C_2)=U_i(C_1)U_i(C_2).
$$
This finishes the proof.
\end{proof}

\begin{lem}\label{lem:DEP}
Let $D$ be a Dedekind domain, let $G$ be a reductive group over $D$, let $P$ be a strictly proper parabolic $D$-subgroup
of $G$. Assume that $(G/P)(D)=G(D)/P(D)$ and $G(D)=E_P(D)P(D)$. Then $G(D_S)=E_P(D)\cdot P(D_S)$ and
$$
E_P(D_S)=E_P(D)\cdot (E_P(D_S)\cap P(D_S))
$$
for any multiplicatively closed subset $S$ of $D$.
\end{lem}
\begin{proof}
Let $K$ be the fraction field of $D$. Since $D$ is a Dedekind domain and $G/P$ is a smooth projective $D$-scheme,
we have $(G/P)(D)=(G/P)(K)=(G/P)(D_S)$ (see e.g.~\cite[Corollaire 7.3.6]{EGAII}).
Since $(G/P)(D)=G(D)/P(D)$, we conclude that $G(D_S)=G(D)P(D_S)$.
Since $G(D)=E_P(D)\cdot P(D)$ by assumption, it follows that $G(D_S)=E_P(D)\cdot P(D_S)$.
Then also $E_P(D_S)=E_P(D)\cdot (E_P(D_S)\cap P(D_S))$.

\end{proof}

\begin{lem}\label{lem:EE}
Let $B\subseteq A$ be two commutative rings, let $h\in B$ be such that $h$ is a non-zero divisor in $A$ and
$B/hB\cong A/hA$. Let $G$ be a reductive group scheme over $B$, let $P,P^-$ be two opposite strictly proper parabolic
$B$-subgroups
of $G$. Assume that
\begin{equation}\label{eq:EBh-0}
E_P(B_h)=E_P(B)\cdot (E_P(B_h)\cap P(B_h))=E_P(B)\cdot (E_P(B_h)\cap P^-(B_h)).
\end{equation}
Then $E_P(A_h)=E_P(A)E_P(B_h)$.
\end{lem}
\begin{proof}
Let $P^-$ be any parabolic $B$-subgroup of $G$ opposite to $P$.
Let $L =P \cap P^{-}$ be the Levi subgroup of $P$.
To prove the claim of the Lemma, it is enough to prove that
\begin{equation}\label{eq:EBhUQ}
E_{P}(B_h) \cdot U_Q(A_h)\subseteq E_{P}(A)\cdot E_{P}(B_h),
\end{equation}
where $Q$ is one of $P$, $P^-$. Indeed, we have $E_P(A_h)=\left<U_P(A_h),\, U_{P^-}(A_h)\right>$, so any fixed $g\in E_P(A_h)$
has a presentation $g=u_1u_2\ldots u_n$, $u_i\in U_P(A_h)$ or $u_i\in U_{P^-}(A_h)$. Proceeding by induction on $n$, we deduce
from~\eqref{eq:EBhUQ} that
$$
g\in E_P(A)\cdot E_P(B_h)\cdot U_Q(A_h)\subseteq E_P(A)\cdot E_{P}(A)\cdot E_{P}(B_h)=E_P(A)\cdot E_P(B_h),
$$
as required.

To prove~\eqref{eq:EBhUQ}, we start by applying~\eqref{eq:EBh-0} and obtaining
\begin{equation}\label{eq:EBh2}
E_{P}(B_h)=E_{P}(B)\cdot (E_{P}(B_h)\cap Q(B_h))\subseteq
E_{P}(A)\cdot  (E_{P}(B_h)\cap L(B_h))\cdot U_Q(B_h).
\end{equation}
Since $L(B_h)\le L(A_h)$ normalizes $U_Q(A_h)$, we deduce
that
\begin{equation}\label{eq:EQL}
E_{P}(B_h) \cdot U_Q(A_h)\subseteq E_{P}(A)\cdot  (E_{P}(B_h)\cap L(B_h))\cdot U_Q(A_h)
\subseteq E_{P}(A)\cdot U_Q(A_h)\cdot (E_{P}(B_h)\cap L(B_h)).
\end{equation}
By the choice of $B\subseteq A$, we have $A=B+hA$. Replacing the $A$ on the right-hand side of the latter equality by $B+hA$,
we deduce that $A=B+h^nA$ for any $n\ge 1$. Since $h$ is a non-zero divisor, $B$, $B_h$ and $A$ are subrings of $A_h$,
and the previous equality implies $A_h=B_h+A$. Then by Lemma~\ref{lem:Ui}
we have
$$
U_Q(A_h)=U_Q(A)\cdot U_Q(B_h).
$$
Substituting this equality into~\eqref{eq:EQL}, we obtain~\eqref{eq:EBhUQ}. This finishes the proof.
\end{proof}

\begin{cor}\label{cor:Efg}
Let $k$ be a field, let $G$ be a reductive algebraic group over $k$, and
let $P$ be a strictly proper parabolic subgroup of $G$. Then for any two coprime polynomials $f,g\in k[x]$
one has
$E_P(k[x]_{fg})=E_P(k[x]_f)\cdot E_P(k[x]_g)$.
\end{cor}
\begin{proof}
We check that Lemma~\ref{lem:DEP} applies to $G$ over $k[x]$.
By the Margaux-Soul\'e theorem~\cite{M} we have $G(k[x])=G(k)\cdot E_{P}(k[x])$. By~\cite[Exp. XXVI, Th. 5.1]{SGA3} $G(k)=E_P(k)\cdot P(k)$,
hence $G(k[x])=E_P(k[x])\cdot P(k)$. Let $L$ be a Levi subgoup of $P$. By~\cite[Prop. 2.2]{CTO}
the map $H^1_{\et}(k[x],L)\to H^1_{\et}(k(x),L)$ has trivial kernel. By~\cite[Exp. XXVI, Cor. 5.10]{SGA3} the map
$H^1_{\et}(k(x),L)\to H^1_{\et}(k(x),G)$ is injective. Hence $H^1_{\et}(k[x],L)\to H^1_{\et}(k[x],G)$ has trivial kernel.
Since $H^1_{\et}(k[x],P)=H^1_{\et}(k[x],L)$,
 the ``long'' exact sequence of \'etale cohomology associated to $1\to P\to G\to G/P\to 1$
then implies that $(G/P)(k[x])=G(k[x])/P(k[x])$. Then all the conditions of Lemma~\ref{lem:DEP} are satisfied for
$G$ over $k[x]$, and hence
$$
E_P(k[x]_f)=E_P(k[x])\cdot (E_P(k[x]_f)\cap P(k[x]_f)).
$$

Obviously, the same argument applies to any opposite parabolic subgroup $P^-$ of $P$.
Now we see that Lemma~\ref{lem:EE} applies to $G$ with $B=k[x]$, $A=k[x]_g$ and $h=f$. It follows that
$E_P(k[x]_{fg})=E_P(k[x]_f)E_P(k[x]_g)$.
\end{proof}

\begin{cor}\label{cor:DD'}
Let $B\subseteq A$ be two discrete valuation rings with a common uniformizer
$h\in B$, and such that
$B/hB\cong A/hA$. Let $G$ be a reductive group scheme over $B$ having a strictly proper parabolic
$B$-subgroup scheme $P$. Then $E_P(A_h)=E_P(A)\cdot E_P(B_h)$.
\end{cor}
\begin{proof}
As in the proof of Corollary~\ref{cor:Efg}, it is enough to check that the conditions
of Lemma~\ref{lem:DEP} hold for $G$ over $B$, and then apply Lemma~\ref{lem:EE}.
Since $B$ is a local ring, one has $G(B)=E_P(B)\cdot P(B)$
by~\cite[Exp. XXVI, Th. 5.1]{SGA3}.
Also, by~\cite[Exp. XXVI, Cor. 5.10]{SGA3} the map
$H^1_{\et}(B,P)\to H^1_{\et}(B,G)$ is injective. Hence
 the ``long'' exact sequence of \'etale cohomology associated to $1\to P\to G\to G/P\to 1$
implies that $(G/P)(B)=G(B)/P(B)$. Then all the conditions of Lemma~\ref{lem:DEP} are satisfied.
\end{proof}

\begin{lem}\label{lem:k[x]p}
Let $B=k[x]_p$ be a discrete valuation ring which is a local
ring of an affine line over a field $k$, let $K=k(x)$ be its fraction field.
Let $G$ be a simply connected semisimple algebraic group over $k$,
let $P$ be a strictly proper parabolic subgroup of $G$. Then
we have isomorphisms
$$
K_1^{G,P}(k) \cong K_1^{G,P}(k[x]_p) \cong  K_1^{G,P}(k(x)).
$$
\end{lem}
\begin{proof}
The isomorphism $K_1^{G,P}(k) \cong   K_1^{G,P}(k(x))$
is \cite[Th. 5.8]{Gi}. Hence the map $K_1^{G,P}(k[x]_p)  \to K_1^{G,P}(k(x))$ is onto, and
$$
K_1^{G,P}(k[x]_p)=K_1^{G,P}(k) \oplus \ker\bigl( K_1^{G,P}(k[x]_p) \to
K_1^{G,P}(k(x))\bigr).
$$
Let us establish the triviality of the kernel.
This reduces to proving that
$$
K_1^{G,P}(k[x]_g)\to K_1^{G,P}(k[x]_{gf})
$$ has trivial kernel,
where $f,g\in k[x]$ are
coprime polynomials. Take $a \in G(k[x]_g)\cap E_{P}(k[x]_{gf})$. We need to show $a\in E_P(k[x]_g)$.
By Corollary~\ref{cor:Efg} we have that
$$
E_{P}(k[x]_{gf})=E_{P}(k[x]_g)\cdot E_P(k[x]_f).
$$
Multiplying $a$ by a suitable element of $E_P(k[x]_g)$, we then have $a\in G(k[x]_g)\cap E_P(k[x]_f)$.
Since $G(k[x]_f)\cap G(k[x]_g)=G(k[x])$, we conclude that $a\in G(k[x])\cap E_P(k[x]_f)$.
By the Margaux-Soul\'e theorem~\cite{M} we have $G(k[x])=G(k)\cdot E_{P}(k[x])$. Therefore
\begin{equation}\label{eq:aMar}
a\in G(k)\cdot E_P(k[x])\cap E_P(k[x]_f).
\end{equation}
If $k$ is finite, then $G$ is a quasi-split simply connected group over $k$ and $G(k)=E_P(k)$, so $a\in E_P(k[x])\subseteq E_P(k[x]_g)$,
as required. If $k$ is infinite, then there is $u\in k$ such that $f(u)\in k^\times$.
Then, taking $x=u$ in~\eqref{eq:aMar}, we see that $a|_{x=u}\in E_P(k)$, and hence
$a\in E_P(k)\cdot E_P(k[x])=E_P(k[x])\subseteq E_P(k[x]_g)$ as well.
\end{proof}

\section{Proof of the main results}

\begin{proof}[Proof of Theorem~\ref{thm:surj}]
Let $\hat D$ be the complete discrete valuation ring obtained by completing $D$ with respect to the maximal ideal, and let $\hat K$ be the fraction field of $\hat D$.
Let $h\in D$ be a common uniformizer of $D$ and $\hat D$. Then $K=D_h$ and $\hat K=\hat D_h$.
Let $g\in G(K)$ be any element.
Since $\hat D$ is henselian, the natural map $K_1^{G,P}(\hat D)\to K_1^{G,P}(\hat K)$ is surjective by~\cite[Prop. 8.10]{GSt}.
Hence $g\in G(\hat D)\cdot E_P(\hat K)$. By
Corollary~\ref{cor:DD'} we have $E_P(\hat K)=E_P(K)\cdot E_P(\hat D)$.
Multiplying $g\in G(K)$ by a suitable element of $E_P(K)$, we then achieve that $g\in G(\hat D)\le G(\hat K)$.
Since $G(\hat D)\cap G(K)=G(D)$, we conclude that $g\in G(D)$. This proves that $G(K)=G(D)\cdot E_P(K)$,
as required.
\end{proof}

\begin{lem}\label{lem:inj-sc}
Let $D$ be any discrete valuation ring, let $K$ be the fraction field of $D$, and
let $G$ be a reductive group scheme over $D$ having a strictly proper parabolic
$D$-subgroup scheme $P$. Let $G^{sc}$ be the simply connected cover of the derived subgroup $G^{der}$ of $G$ over $D$.
Then $G^{sc}$ has a strictly proper parabolic subgroup $P^{sc}$, and the canonical
homomorphism $G^{sc}(D)\to G^{der}(D)\to G(D)$ induces a surjection from the kernel of the map
\begin{equation}\label{eq:Gsc}
K_1^{G^{sc},P^{sc}}(D)\to K_1^{G^{sc},P^{sc}}(K)
\end{equation}
onto the kernel of
the map
\begin{equation}\label{eq:GDK}
K_1^{G,P}(D)\to K_1^{G,P}(K).
\end{equation}
\end{lem}
\begin{proof}
The intersection $P^{der}=G^{der}\cap P$ is a strictly proper parabolic subgroup of $G^{der}$ by~\cite[Exp. XXVI, Prop. 1.19]{SGA3}.
Let $\pi:G^{sc}\to G^{der}$ be the canonical homomorphism, then $P^{sc}=\pi^{-1}(P^{der})$ is a strictly proper
parabolic subgroup of $G^{sc}$.

There is a short exact sequence of algebraic groups
$$
1\to C\xrightarrow{i} G^{sc}\xrightarrow{\pi} G^{der}\to 1,
$$
where $C$ is a finite group of multiplicative type over $D$, central in $G^{sc}$. Write the respective ``long'' exact
sequences over $D$ and $K$
with respect to fppf topology. We obtain a commutative diagram
\begin{equation*}
\xymatrix@R=15pt@C=20pt{
1\ar[r]\ar@{=}[d]&C(D)\ar[d]\ar[r]^{i}&G^{sc}(D)\ar[d]\ar[r]^\pi &G^{der}(D)\ar[d]\ar[r]^{\hspace{-15pt}\delta}&H^1_{fppf}(D,C)\ar[d]\\
1\ar[r]&C(K)\ar[r]^{i}&G^{sc}(K)\ar[r]^\pi&G^{der}(K)\ar[r]^{\hspace{-15pt}\delta_K}&H^1_{fppf}(K,C)\\
}
\end{equation*}
Here the rightmost vertical arrow is injective by~\cite[Th. 4.1]{CTS}. Take any
$g\in G^{der}(D)\cap E_{P^{der}}(K)$. Then $\delta_K(g)=1$, since $E_{P^{sc}}(K)$ surjects onto $E_{P^{der}}(K)$.
Hence there is $\tilde g\in G^{sc}(D)$ with
$\pi(\tilde g)=g$. Clearly, $\tilde g\in C(K)\cdot E_{P^{sc}}(K)$, since $\pi(\tilde g)\in E_{P^{der}}(K)$. However,
since $D$ is a discrete valuation ring and $C$ is finite, we have $C(D)=C(K)$. (Note that $C$ embeds into into a quasi-split
$D$-torus which is a maximal torus of the unique quasi-split inner $D$-form of $G^{sc}$~\cite[Exp.\ XXIV, Proposition 3.13]{SGA3}, and hence $C$ embeds into some
split $D$-torus.) Therefore,
$$
\tilde g\in G^{sc}(D)\cap C(D)\cdot E_{P^{sc}}(K)= C(D)\cdot (G^{sc}(D)\cap E_{P^{sc}}(K)).
$$
Since $C(D)\le\ker\pi$,
it follows that $G^{sc}(D)\cap E_{P^{sc}}(K)$ surjects onto $G^{der}(D)\cap E_{P^{der}}(K)$.

Consider now the short exact sequence $1\to G^{der}\to G\to \corad(G)\to 1$~\cite[Exp. XXIII, 6.2.3]{SGA3}.
Taking into account that $\corad(G)(D)\to \corad(G)(K)$ is injective and $E_P(D)=E_{P^{der}}(D)$, $E_P(K)=E_{P^{der}}(K)$,
we immediately see that $G(D)\cap E_P(K)= G^{der}(D)\cap E_{P^{der}}(K)$. Therefore,
 $G^{sc}(D)\cap E_{P^{sc}}(K)$ surjects onto $G(D)\cap E_P(K)$, which is exactly the claim of the lemma.
\end{proof}

\begin{proof}[Proof of Theorem~\ref{thm:main}]
By Lemma~\ref{lem:inj-sc} we can assume that $G$ is simply connected semisimple.
If $k$ is a finite field, then $G$ is a quasi-split simply connected semisimple group, and since $D$ is local,
we have $K_1^{G,P}(D)=1$ and there is nothing to prove. Assume $k$ is infinite. We may assume that $D$
is a local ring of a smooth irreducible algebraic curve without loss of generality. By Ojanguren's lemma~\cite[Lemme 1.2]{CTO}
there is a maximal localization $k[x]_p$ of the polynomial ring $k[x]$ at a maximal ideal $p=(f)$ and an essentially \'etale
local homomorphism
$\phi:k[x]_p\to D$ such that $\phi(f)$ is a uniformizer of $D$ and the induced map $k[x]_p/f\cdot k[x]_p\to D/\phi(f)\cdot D$
is an isomorphism. Then the triple $B=k[x]_p$, $A=D$ and $h=f$ is subject to Corollary~\ref{cor:DD'} (note that
$\phi$ is injective, since it is flat). Hence $E_P(K)=E_P(D_{\phi(f)})=E_P(D)\cdot E_P(\phi(k(x)))$. Assume that $g\in G(D)$ is mapped into
$E_P(K)\le G(K)$. Then, multiplying $g$ by an element of $E_P(D)$, we achieve that $g\in G(D)\cap E_P(\phi(k(x)))=
\phi(G(k[x]_p)\cap E_P(k(x)))$.
By Lemma~\ref{lem:k[x]p} this implies that $g\in \phi(E_P(k[x]_p))\le E_P(D)$. Then $g\in E_P(D)$, and
the injectivity of $K_1^{G,P}(D)\to K_1^{G,P}(K)$ is proved.
\end{proof}

\section{Proof of Corollary~\ref{cor:sep}}

The following two lemmas are very standard, however, we were unable to find perfectly matching references.

\begin{lem}\label{lem:locmax-1}
Let $k$ be a field, and let $A$ be a finitely generated $k$-algebra. Let $D$ be a local ring of $A$. Then there is
a purely transcendental field extension $l=k(x_1,\ldots,x_n)$ of $k$  such that $D$ is
a localization of a finitely generated $l$-algebra at a maximal ideal. Moreover, if the residue field of $D$ is separable
over $k$, we may secure that it is separable over $l$.
\end{lem}
\begin{proof}
Let $q\subset A$ be a prime ideal such that $D=A_q$, and let $L$ be the residue field of $D$. We have $L=\Frac(A/q)$.
Let $L'=k(t_1,\ldots,t_n)$ be a purely transcendental field extension of $k$ such that $k\subseteq L'\subseteq L$
and $L$ is a finite extension of $L'$, which is also separable if $L$ is separable over $k$.
Let $x_1,\ldots,x_n\in A$ be any lifts of $t_1,\ldots,t_n$.
Then $l=k(x_1,\ldots,x_n)$ is a purely transcendental field extension of $k$ that embeds into $D=A_q$ and is mapped
isomorphically onto $L'$ inside $L=A_q/qA_q$. Then
$A'=A\otimes_{k[x_1,\ldots,x_n]} l$ is a finitely generated $l$-algebra such that $D$ is a localization of $A'$
at a prime ideal $q'$. Moreover, $L$ is a finite extension of $l\cong L'$, so, in particular, $q'$ is a maximal
ideal of $A'$.
\end{proof}

\begin{lem}\label{lem:locmax-2}
Let $k$ be a field, and let $A$ be a finitely generated $k$-algebra. Let $D$ be a local ring of $A$ such that $D$ is regular
and the residue field $L$ of $D$ is separable over $k$. Then there is
a purely transcendental field extension $l=k(x_1,\ldots,x_n)$ of $k$ such that
$D$ is a localization of an integral smooth $l$-algebra at a maximal ideal.
\end{lem}
\begin{proof}
Let $q\subset A$ be a prime ideal such that $D=A_q$.
By Lemma~\ref{lem:locmax-1} we can assume without loss of generality that $q$ was a maximal ideal of $A$.
Since $D$ is a regular ring and its residue field $L$ is separable over $k$, and $A$ is of finite type over $k$,
by~\cite[Tag 00TV]{Stacks}
$A$ is smooth over $k$ at $q$. By definition, it means that there is $g\in A$, $g\not\in q$, such that
$k\to A_g$ is a smooth ring map (in particular, $A_g$ is of finite type over $k$). Since $A_g$ is smooth over $k$,
by~\cite[Tag 00TT]{Stacks} the $k$-algebra $A_g$ is a regular ring. Hence $A_g$ is a finite direct sum of regular domains,
and there is  $f\in A$ such that $f\not\in q$ and $A_{gf}$ is a regular domain, so that $D$ is a localization of $A_{gf}$
at $qA_{gf}$. Thus, $D$ is a localization of an integral smooth $k$-algebra $A_{gf}$ at a maximal ideal.
\end{proof}

\begin{proof}[Proof of Corollary~\ref{cor:sep}]
Follows immediately from Theorem~\ref{thm:main} and Lemma~\ref{lem:locmax-2}.
\end{proof}

\renewcommand{\refname}{References}

\end{document}